\documentclass[11pt,reqno]{amsart}

\usepackage[margin=1in]{geometry}
\usepackage{amsmath,amssymb,amsthm,mathtools}
\usepackage{enumitem}
\usepackage{booktabs,array}
\usepackage{microtype}
\usepackage[T1]{fontenc}
\usepackage{lmodern}
\usepackage[hidelinks]{hyperref}
\numberwithin{equation}{section} 

\newtheorem{theorem}{Theorem}[section]
\newtheorem{proposition}[theorem]{Proposition}
\newtheorem{lemma}[theorem]{Lemma}
\newtheorem{corollary}[theorem]{Corollary}
\newtheorem{remark}[theorem]{Remark}
\newtheorem{definition}[theorem]{Definition}

\newcommand{\R}{\mathbb R}
\newcommand{\Hog}[3]{\left\langle \genfrac{}{}{0pt}{}{#1}{#2}\right\rangle^{(#3)}}
\newcommand{\qge}{\geq_t}

\newcommand{\coef}{\succeq_t}

\DeclareMathOperator{\diag}{diag}

\makeatletter
\def\paragraph{\@startsection{paragraph}{4}%
  \z@\z@{-\fontdimen2\font}%
  {\normalfont\bfseries}}
\makeatother

\title[Narayana Interpolations and Total Positivity]{Continuous Narayana Interpolations:\\
Hankel Total Positivity and Log-Convexity Preservation}
\author{Jiang Zeng}
\address{College of Mathematics and Physics, Wenzhou University, Wenzhou 325035, China}
\address{Universit\'e Claude Bernard Lyon 1, CNRS UMR 5208, Institut Camille Jordan, France}
\email{zeng@math.univ-lyon1.fr}
\date{\today}
\begin{document}
\begin{abstract}
We study a continuous interpolation between the type $B$ and shifted
type $A$ Narayana triangles. Using a direct Gegenbauer-moment proof and
a continuant comparison criterion, we show that the row polynomials
form a coefficientwise Hankel-totally positive sequence for every
real $\theta\ge0$, and establish the same property for the moment sequences associated with all tails of the corresponding Jacobi matrix. Their
$\gamma$-expansion, real-rootedness and Stieltjes moment representation
are explicit, as are the exact orders of vanishing and extreme
coefficients of arbitrary Hankel minors.
This confirms and strengthens Rodelet--Causse--Tevlin's
log-convexity conjecture for super Narayana polynomials.
We further extend Wang--Yang's total positivity of the type~$A$ and
$B$ Narayana triangles and strict total positivity of the associated
squares to every real $\theta>-1$, using P\'olya-frequency sequences
and a gamma-product moment representation.
We also prove that the continuous Narayana transformations preserve
strong $t$-log-convexity for every $\theta\ge0$, and that every
$d$-Hoggatt transformation has the same property for integers $d\ge1$.
The preservation proofs employ the Zhu--Sun criterion and a
mixed-minor one-sign-change argument based on a hypergeometric
hitting-probability inequality. 
\end{abstract}
\maketitle

\section{Introduction}
Two classical Narayana triangles are the type~$B$ triangle
$\binom nk^2$ and the shifted type~$A$ triangle
$N(n+1,k+1)=\frac{n+1}{(k+1)(n-k+1)}\binom nk^2$.
Zhu and Sun~\cite{ZhuSun2015} proved strong
$t$-log-convexity preservation for the type~$B$ transformation and
for the ordinary (non-shifted) type~$A$ Narayana transformation.
The precise relation between their latter formulation and our
shifted type~$A$ kernel is explained in
Remark~\ref{rem:ZhuSun_index_shift}.
For a real parameter $\theta\ge0$, define the continuous interpolation
\begin{equation}\label{eq:interpolating_triangle_intro}
 \mathcal N_\theta(n,k)=\binom nk
 \frac{(\theta+1)_n}{(\theta+1)_k(\theta+1)_{n-k}},
 \qquad 0\le k\le n,
\end{equation}
where $(a)_n$ denotes the rising factorial. In particular,
$\mathcal N_0(n,k)=\binom nk^2$ and
$\mathcal N_1(n,k)=N(n+1,k+1)$.
Its row polynomials are
\begin{equation}\label{eq:theta_rows_intro}
 P_n^{(\theta)}(t)=\sum_{k=0}^n\mathcal N_\theta(n,k)t^k.
\end{equation}
For integral $\theta$, these hypergeometric polynomials are related to
recent work of Mao and Wang~\cite{MaoWang2026}.
Wang and Yang~\cite{WangYang2018} proved total positivity of the
classical type~$A$ and type~$B$ Narayana triangles and strict total
positivity of their associated squares.  Our interpolation unifies
these two endpoint statements, and their methods extend to the full
continuous family:
\begin{theorem}\label{thm:wangyang_intro}
For every real $\theta>-1$, the lower-triangular matrix
$[\mathcal N_\theta(n,k)]_{n,k\ge0}$ is totally positive, and the
square $[\mathcal N_\theta(i+j,j)]_{i,j\ge0}$ is strictly totally
positive.
\end{theorem}
The proof in Section~\ref{sec:wangyang_classical} uses P\'olya-frequency
sequences and Stieltjes moments, as in~\cite{WangYang2018}.
These matrix statements should be distinguished from coefficientwise
Hankel total positivity of the row polynomials, which is the next result.

We next establish an all-orders positivity theorem for the row-generating polynomials, viewed as a moment sequence rather than as the entries of a triangle.
\begin{theorem}\label{thm:theta_structure_intro}
For every real $\theta\ge0$, the sequence
$\{P_n^{(\theta)}(t)\}_{n\ge0}$ and every associated Jacobi tail are
coefficientwise Hankel-totally positive. Moreover,
\begin{equation}\label{eq:gamma_intro}
P_n^{(\theta)}(t)=
\sum_{j=0}^{\lfloor n/2\rfloor}
\frac{n!}{j!(n-2j)!(\theta+1)_j}
 t^j(1+t)^{n-2j}.
\end{equation}
Thus these row polynomials are $\gamma$-positive. For each fixed
$t\ge0$, they form a Stieltjes moment sequence with support in
$[(1-\sqrt t)^2,(1+\sqrt t)^2]$.
\end{theorem}

The classical ultraspherical (Gegenbauer) probability measure is
\begin{equation}\label{eq:classical_moment_bridge_intro}
 d\nu_\theta(x)=
 \frac{\Gamma(\theta+1)}{\sqrt\pi\,\Gamma(\theta+\tfrac12)}
 (1-x^2)^{\theta-1/2}\,dx \quad (-1<x<1),
\end{equation}
with the limiting arcsine measure at $\theta=0$. Its affine moments
are precisely
\begin{equation}\label{eq:classical_affine_intro}
 P_n^{(\theta)}(t)=\int_{-1}^1(1+t+2\sqrt t\,x)^n\,d\nu_\theta(x).
\end{equation}
The Hankel theorem will be proved \emph{directly} from the Gegenbauer
three-term recurrence, without taking limits of basic-hypergeometric
identities. In particular the proof works within the classical theory.
The method uses a comparison of tridiagonal continuants and the
production-matrix factorization of Hankel matrices; see
\cite{ChenSokal2024,PSZ,Zhu2024,DDPS} for related transfer theorems.
We also determine the exact vanishing orders, palindromicity, and
extreme coefficients of arbitrary Hankel minors.

At integers $s\ge0$, the classical super Narayana polynomials of
Rodelet--Causse and Tevlin~\cite{RodeletTevlin2025} satisfy
$\mathcal N_{s+r}^{(s)}(1,t)=\binom{2s}sP_r^{(s)}(t)$ for $r\ge0$.
Theorem~\ref{thm:theta_structure_intro} therefore confirms their
Conjecture~3 on coefficientwise $t$-log-convexity and strengthens
it to Hankel total positivity of all orders; see
Corollary~\ref{cor:super-Narayana-HTP}.

Our other two theorems concern positivity-preserving transformations,
which are different from Hankel positivity of the particular row sequence.
\begin{theorem}\label{thm:interpolation}
Fix $\theta\ge0$. If $\{f_k(t)\}_{k\ge0}$ is strongly
$t$-log-convex with nonnegative coefficients, then
\[
 G_n^{(\theta)}(t)=\sum_{k=0}^n\mathcal N_\theta(n,k)f_k(t)
\]
is strongly $t$-log-convex.
\end{theorem}
For every integer $d\ge1$, put
\begin{equation}\label{eq:hoggatt}
\Hog{n}{k}{d}=\prod_{j=0}^{d-1}
\frac{\binom{n+j}{k}}{\binom{k+j}{k}}.
\end{equation}
These are the $d$-Hoggatt numbers; $d=2$ recovers the shifted type~$A$
Narayana triangle. Chern and Shi~\cite{ChernShi2026} proved a related
ordinary log-convexity result and strong $t$-log-convexity of their
row-generating polynomials.
The following is a coefficientwise strengthening 
of Chern and Shi's log-convexity-preservation theorem.

\begin{theorem}\label{thm:main}
For every integer $d\ge1$, the Hoggatt transform
\begin{equation}\label{eq:hoggatt_q_transform}
 G_n(t)=\sum_{k=0}^n\Hog{n}{k}{d}f_k(t)
\end{equation}
preserves strong $t$-log-convexity of polynomial sequences with
nonnegative coefficients.
\end{theorem}

The preservation theorems  use the Zhu--Sun criterion and a mixed-minor
one-sign-change argument, whose key local ingredient is a
hypergeometric hitting-probability inequality. Neither transformation
theorem follows merely from the Hankel total positivity of the rows.

Section~\ref{ar:sec:criterion} supplies a continuant criterion.
Section~\ref{ar:sec:limits} proves the Gegenbauer Hankel theorem
and the Wang--Yang generalization for triangles and squares,
Section~\ref{ar:sec:extreme} treats arbitrary Hankel minors, and
Sections~\ref{sec:prelim}--\ref{sec:theta_preservation} prove the
preservation statements. The remaining sections discuss applications
and open questions. The two basic-hypergeometric deformations of the
classical construction are treated together in a separate paper
\cite{ZengQUltraspherical}.

\section{A leading-continuant criterion}\label{ar:sec:criterion}
For $n\ge1$, let
\begin{equation}\label{ar:eq:continuant}
 C_n(b_1,\ldots,b_{n-1};t)=
 \det\begin{pmatrix}
 1+t&1&&&\\ tb_1&1+t&1&&\\ &tb_2&1+t&\ddots&\\ &&\ddots&\ddots&1\\ &&&tb_{n-1}&1+t
 \end{pmatrix},
\end{equation}
with $C_0=1$ and $C_1=1+t$. Empty lists will be suppressed.
Consider the path on vertices $1,\ldots,n$, with edge $i$ joining $i$ to
$i+1$. A matching is a set of edges with no shared vertices, equivalently
a subset of $\{1,\ldots,n-1\}$ containing no consecutive indices.
In the determinant expansion, every nonzero term corresponds to fixed
points and disjoint adjacent transpositions. Each fixed point contributes
$1+t$, while the transposition $(i,i+1)$ contributes $-tb_i$,
including its permutation sign. Thus the expansion over matchings gives
\begin{equation}\label{ar:eq:matching}
 C_n(b;t)=\sum_M(-t)^{|M|}\prod_{i\in M}b_i\,(1+t)^{n-2|M|},
\end{equation}
where $M$ ranges over these matchings. For example,
\[
 C_4(b_1,b_2,b_3;t)=(1+t)^4
 -t(b_1+b_2+b_3)(1+t)^2+t^2b_1b_3,
\]
where the last term corresponds to the matching consisting of edges $1$
and $3$. It follows that
\begin{equation}\label{ar:eq:derivative}
 \frac{\partial C_n}{\partial b_i}
 =-t\,C_{i-1}(b_1,\ldots,b_{i-2};t)
 C_{n-i-1}(b_{i+2},\ldots,b_{n-1};t).
\end{equation}
Both identities also follow by expanding the determinant along the edge indexed by $i$.

\begin{lemma}[Telescoping identity for continuants]\label{ar:lem:telescoping}
For arbitrary real weights $a_i,b_i$ and $n\ge2$,
\begin{align}
 C_n(b_1,\ldots,b_{n-1};t)
 &=C_n(a_1,\ldots,a_{n-1};t)\notag\\
 &\quad+t\sum_{i=1}^{n-1}(a_i-b_i)
 C_{i-1}(a_1,\ldots,a_{i-2};t)
 C_{n-i-1}(b_{i+2},\ldots,b_{n-1};t).
 \label{ar:eq:telescoping}
\end{align}
Empty lists are interpreted using $C_0=1$ and $C_1=1+t$.
\end{lemma}
\begin{proof}
Replace $a_{n-1},a_{n-2},\ldots,a_1$ successively by
$b_{n-1},b_{n-2},\ldots,b_1$ in the continuant $C_n$.
Here a continuant is a tridiagonal determinant, in this case the determinant
defined in \eqref{ar:eq:continuant}. By the matching expansion
\eqref{ar:eq:matching}, each weight occurs at most once in every term.
Thus $C_n$ is affine in each weight: with all other weights fixed,
it has degree at most one in that weight. Replacing $a_i$ by $b_i$
therefore changes the determinant by exactly $(b_i-a_i)$ times its
partial derivative with respect to the $i$th weight. Differentiation
with respect to edge $i$ gives
\[
 \frac{\partial}{\partial a_i}
 C_n(a_1,\ldots,a_i,b_{i+1},\ldots,b_{n-1};t)
 =-t C_{i-1}(a_1,\ldots,a_{i-2};t)
 C_{n-i-1}(b_{i+2},\ldots,b_{n-1};t).
\]
At the replacement of edge $i$ the left weights still equal $a$,
whereas the right weights already equal $b$. Summing the changes
proves the identity.
\end{proof}

\begin{theorem}[Leading-continuant criterion]\label{ar:thm:criterion}
Let $a_1\ge a_2\ge\cdots\ge0$, and suppose that
\begin{equation}\label{ar:eq:leadinghyp}
 C_n(a_1,\ldots,a_{n-1};t)\coef0\qquad(n\ge0).
\end{equation}
Then for every $n$ and every $0\le b_i\le a_i$,
\begin{equation}\label{ar:eq:box}
 C_n(b_1,\ldots,b_{n-1};t)
 \coef C_n(a_1,\ldots,a_{n-1};t)\coef0.
\end{equation}
Consequently, if an infinite sequence $b$ satisfies $0\le b_i\le a_i$, the tridiagonal matrix
\begin{equation}\label{ar:eq:Jb}
 J_b(t)=\begin{pmatrix}
 1+t&1&&\\ tb_1&1+t&1&\\ &tb_2&1+t&\ddots\\ &&\ddots&\ddots
 \end{pmatrix}
\end{equation}
is coefficientwise totally positive.
\end{theorem}
\begin{proof}
We induct on $n$, proving the assertion for the entire box
$\prod_{i=1}^{n-1}[0,a_i]$ at each size. The cases $n=0,1$ are
immediate. Apply Lemma~\ref{ar:lem:telescoping}. The left factor in each
summand of \eqref{ar:eq:telescoping} is a leading continuant and is
coefficientwise nonnegative by \eqref{ar:eq:leadinghyp}. For the right
factor, its $j$th weight satisfies
\[
 0\le b_{i+1+j}\le a_{i+1+j}\le a_j.
\]
Its size is smaller than $n$, so the induction hypothesis makes it
coefficientwise nonnegative. Since $a_i-b_i\ge0$, every summand is
coefficientwise nonnegative. Thus \eqref{ar:eq:telescoping} proves
\eqref{ar:eq:box} and provides a recursive positive certificate for it.
The induction concerns all lower-size boxes and does not assume
positivity of the shifted continuants being proved.

A contiguous principal block beginning at row $s$ has determinant
$C_n(b_{s+1},\ldots,b_{s+n-1};t)$. Its $i$th edge weight is bounded by $a_{s+i}\le a_i$, so \eqref{ar:eq:box} applies. All these determinants are nonnegative coefficientwise, as are the entries of $J_b$. The standard tridiagonal criterion now gives total positivity: all minors of a tridiagonal matrix are nonnegative if its entries and its contiguous principal minors are nonnegative. This criterion is valid over a commutative partially ordered ring; see \cite[Lemma 2.4 and Corollary 2.5]{DDPS}. Equivalently, successive Laplace expansions factor each nonzero minor into off-diagonal entries and contiguous principal determinants. The expansions use only products and sums, so the criterion applies to coefficientwise order.
\end{proof}

\begin{corollary}\label{ar:cor:one}
If $0\le b_i\le1$ for every $i$, then $J_b(t)$ is coefficientwise totally positive.
\end{corollary}
\begin{proof}
For the constant sequence $a_i=1$, the continuant recurrence gives
$C_n(1,\ldots,1;t)=1+t+\cdots+t^n$. Apply Theorem~\ref{ar:thm:criterion}.
\end{proof}

We next recall the production-matrix transfer in a form requiring no divisions.
\begin{proposition}\label{ar:prop:transfer}
Suppose $b_i\ge0$ and $J_b(t)$ is coefficientwise totally positive. Set
\[
 O_{nk}=(J_b(t)^n)_{0k},\qquad B_n=(J_b(t)^n)_{00}.
\]
Then $O$ and the Hankel matrix $(B_{i+j})_{i,j\ge0}$ are coefficientwise totally positive.
\end{proposition}
\begin{proof}
Write $O_N$ for the square output array with row and column indices $0,\ldots,N$. Let $T_N$ be the $(N+2)\times(N+2)$ matrix whose first row is $(1,0,\ldots,0)$ and whose remaining rows are the first $N+1$ rows of $J_b$, restricted to columns $0,\ldots,N+1$. Every minor of $T_N$ either is a minor of $J_b$, is zero, or expands along the first row into a minor of $J_b$ with positive sign. Thus $T_N$ is totally positive. The output recurrence gives
\[
 O_{N+1}=\diag(1,O_N)T_N.
\]
Induction and Cauchy--Binet show that $O_N$ is totally positive for all $N$.

Let $d_0=1$ and $d_k=t^k\prod_{j=1}^k b_j$. Reversal of tridiagonal paths gives
$(J_b^n)_{k0}=d_k(J_b^n)_{0k}$. This identity holds even when some $b_j$ vanish, since it can be verified term by term without division. Matrix multiplication therefore yields
\begin{equation}\label{ar:eq:factorH}
 B_{i+j}=\sum_{k\ge0}O_{ik}d_kO_{jk},\qquad
 H=O\diag(d_0,d_1,\ldots)O^{\mathsf T}.
\end{equation}
The sums relevant to any finite minor are finite: $O_{ik}=0$ for $k>i$. Another application of Cauchy--Binet proves the Hankel assertion.
\end{proof}

\section{Classical Gegenbauer moments and Narayana polynomials}
\label{ar:sec:limits}
The normalized ultraspherical measure $\nu_\theta$ of
\eqref{eq:classical_moment_bridge_intro} has even moments
$\int x^{2j}d\nu_\theta(x)=(\frac12)_j/(\theta+1)_j$ and vanishing
odd moments. Expanding \eqref{eq:classical_affine_intro} gives
\begin{equation}\label{ar:eq:classical}
P_n^{(\theta)}(t)=\sum_{j=0}^{\lfloor n/2\rfloor}
\frac{n!}{j!(n-2j)!(\theta+1)_j}t^j(1+t)^{n-2j}.
\end{equation}
The coefficient of $t^k$ is obtained from Chu--Vandermonde:
\[
\sum_{j\ge0}\frac{(-k)_j(-(n-k))_j}{(\theta+1)_j j!}
 =\frac{(\theta+1+n-k)_k}{(\theta+1)_k}.
\]
Consequently,
\begin{equation}\label{ar:eq:classicalcoeff}
P_n^{(\theta)}(t)=\sum_{k=0}^n
 \binom nk\frac{(\theta+1)_n}{(\theta+1)_k(\theta+1)_{n-k}}t^k.
\end{equation}

\subsection{A direct classical Hankel-positivity proof}
The monic Gegenbauer polynomials in the variable $y=2x$ obey
\[
 p_{n+1}(y)=yp_n(y)-\rho_n(\theta)p_{n-1}(y),\qquad
 \rho_n(\theta)=\frac{n(n+2\theta-1)}{(n+\theta-1)(n+\theta)}
 \quad(n\ge1,\ \theta>0),
\]
where $p_0=1$, $p_1=y$. In particular,
$\rho_1(\theta)=2/(1+\theta)$.
After the affine change $y\mapsto1+t+\sqrt t\,y$,
the Jacobi production matrix has diagonal $1+t$, superdiagonal $1$,
and subdiagonal $t\rho_1,t\rho_2,\ldots$.
Thus, in the notation of Proposition~\ref{ar:prop:transfer},
\begin{equation}\label{ar:eq:classical-J}
 P_n^{(\theta)}(t)=(J_{\rho(\theta)}(t)^n)_{00}.
\end{equation}

\begin{lemma}\label{ar:lem:gegenbauer-leading}
For $\theta>0$, the leading continuant of order $n$ is
\begin{equation}\label{ar:eq:gegenbauer-leading}
 E_n^{(\theta)}(t)=\frac{n!}{(\theta)_n}
 \sum_{k=0}^n\frac{(\theta)_k(\theta)_{n-k}}{k!(n-k)!}t^k,
 \qquad E_0^{(\theta)}=1.
\end{equation}
Every coefficient is strictly positive. As $\theta\downarrow0$,
$E_0^{(0)}=1$ and $E_n^{(0)}=1+t^n$ for $n\ge1$.
\end{lemma}
\begin{proof}
The Gegenbauer generating function
$\sum_{n\ge0}C_n^{(\theta)}(x)z^n=(1-2xz+z^2)^{-\theta}$
gives, on substituting $x=(1+t)/(2\sqrt t)$ and $z=u\sqrt t$,
\[
 \sum_{n\ge0}t^{n/2}C_n^{(\theta)}
 \!\left(\frac{1+t}{2\sqrt t}\right)u^n
 =(1-u)^{-\theta}(1-tu)^{-\theta}.
\]
The monic normalization $n!/(\theta)_n$ yields
\eqref{ar:eq:gegenbauer-leading}. The Gegenbauer three-term recurrence
shows that these polynomials satisfy
$E_0=1$, $E_1=1+t$ and
$E_{n+1}=(1+t)E_n-t\rho_n(\theta)E_{n-1}$;
therefore they are exactly the leading continuants of $J_\rho$.
Coefficient positivity and the limiting formulas are immediate.
\end{proof}

For $r\ge0$, let $J_\theta^{[r]}(t)$ be the Jacobi matrix obtained
from $J_{\rho(\theta)}(t)$ by deleting its first $r$ rows and
columns. Explicitly,
\begin{equation}\label{ar:eq:associated-Jacobi}
 J_\theta^{[r]}(t)=
 \begin{pmatrix}
 1+t & 1 & 0 & \cdots\\
 t\rho_{r+1}(\theta) & 1+t & 1 & \cdots\\
 0 & t\rho_{r+2}(\theta) & 1+t & \ddots\\
 \vdots & \vdots & \ddots & \ddots
 \end{pmatrix},\qquad
 P_n^{(\theta,r)}(t)=\bigl((J_\theta^{[r]}(t))^n\bigr)_{00}.
\end{equation}
At $\theta=0$ we use the limiting values
$\rho_1(0)=2$ and $\rho_j(0)=1$ for $j\ge2$.
We call $\big(P_n^{(\theta,r)}(t)\big)_{n\ge0}$ the $r$th
\emph{associated Jacobi tail}; thus
$P_n^{(\theta,0)}(t)=P_n^{(\theta)}(t)$.
In general, these tails are not the shifted moment sequences
$\big(P_{n+r}^{(\theta)}(t)\big)_{n\ge0}$.

\begin{theorem}\label{ar:cor:classical}
For every $\theta\ge0$ and $r\ge0$, the moment sequence
$\big(P_n^{(\theta,r)}(t)\big)_{n\ge0}$ is
coefficientwise Hankel-totally positive.
\end{theorem}
\begin{proof}
For $\theta\ge1$, a direct calculation gives
\[
 \rho_n(\theta)-1=
 \frac{\theta(1-\theta)}{(n+\theta-1)(n+\theta)}\le0,
\]
so $0<\rho_n(\theta)\le1$. Corollary~\ref{ar:cor:one} makes the
production matrix coefficientwise totally positive.
For $0<\theta<1$ the same expression is positive and strictly
decreasing in $n$; hence
$\rho_1(\theta)>\rho_2(\theta)>\cdots>1$.
Lemma~\ref{ar:lem:gegenbauer-leading} verifies the leading-continuant
hypothesis of Theorem~\ref{ar:thm:criterion}, yielding the same
conclusion. Proposition~\ref{ar:prop:transfer} now proves Hankel total
positivity. Every shifted Jacobi matrix is a submatrix of the original
coefficientwise totally positive production matrix, so the argument
also applies to associated tails. Finally, letting $\theta\downarrow0$
gives the arcsine case with $\rho_1(0)=2$ and
$\rho_n(0)=1$ for $n\ge2$; each fixed minor is a finite polynomial in
finitely many recurrence coefficients, so positivity passes to the
limit. This is a proof entirely within the classical Gegenbauer theory.
\end{proof}

\subsection{The type \texorpdfstring{$A$}{A} and type \texorpdfstring{$B$}{B} endpoints}
At $\theta=1$, $\rho_n=1$ and the measure is the semicircle law.
At $\theta=0$, $\rho_1=2$, $\rho_n=1$ for $n\ge2$, and the measure
is the arcsine law. In particular,
\begin{equation}\label{ar:eq:typeA}
P_n^{(1)}(t)=\sum_{k=0}^n
 \frac{n+1}{(k+1)(n-k+1)}\binom nk^2t^k,
\end{equation}
\begin{equation}\label{ar:eq:typeB}
P_n^{(0)}(t)=\sum_{k=0}^n\binom nk^2t^k.
\end{equation}
\begin{remark}\label{ar:rem:associated-endpoints}
For $\theta=1$ all Jacobi coefficients satisfy $\rho_j(1)=1$,
so every associated tail is the type~$A$ Narayana moment sequence:
$P_n^{(1,r)}(t)=P_n^{(1)}(t)$ for all $r\ge0$.
For $\theta=0$, only the first coefficient differs:
$\rho_1(0)=2$ and $\rho_j(0)=1$ for $j\ge2$.
Consequently, $J_0^{[r]}(t)=J_1^{[0]}(t)$ whenever $r\ge1$, and
\[
 P_n^{(0,r)}(t)=P_n^{(1)}(t)
 \qquad(n\ge0,\ r\ge1).
\]
Thus the first and every subsequent Jacobi tail of the type~$B$
Narayana moments gives the type~$A$ moments.
\end{remark}
The principal determinants are
\begin{equation}\label{ar:eq:classical-det}
 \det[P_{i+j}^{(\theta)}(t)]_{0\le i,j\le m}
 =t^{\binom{m+1}{2}}
 \prod_{h=1}^m\rho_h(\theta)^{m+1-h},
\end{equation}
where the limiting values are used at $\theta=0$.
They specialize to $t^{\binom{m+1}{2}}$ for type~$A$ and
$2^m t^{\binom{m+1}{2}}$ for type~$B$.

\subsection{Total positivity of the interpolating Narayana triangles and squares}
\label{sec:wangyang_classical}
We now record a continuous-parameter version of the type~$A$ and
$B$ Narayana matrix theorems of Wang and Yang~\cite{WangYang2018}.
This follows from the same two structural methods used in their work:
P\'olya-frequency sequences for triangles and Stieltjes moments for
squares.  Here \emph{total positivity} means nonnegativity of all
minors, while \emph{strict total positivity} means positivity of all
minors.  The following statements extend to $\theta>-1$, even though
the Gegenbauer and preservation results elsewhere in this paper are
stated for $\theta\ge0$.

\begin{proof}[Proof of Theorem~\ref{thm:wangyang_intro}]
Put
\[
 w_n=n!(\theta+1)_n,\qquad
 a_n=\frac{1}{w_n}\quad(n\ge0),\qquad a_n=0\quad(n<0).
\]
For the triangle, the defining formula gives a positive diagonal
factorization
\[
 \mathcal N_\theta(n,k)
 =w_n\,a_{n-k}\,a_k,
 \qquad
 [\mathcal N_\theta(n,k)]_{n,k\ge0}
 =\operatorname{diag}(w_n)_{n\ge0}
   [a_{n-k}]_{n,k\ge0}
   \operatorname{diag}(a_k)_{k\ge0}.
\]
For $\theta>-1$, the classical Laguerre--P\'olya product
for the Bessel function gives
\[
 \sum_{n\ge0}a_n z^n
 ={}_0F_1(;\theta+1;z)
 =\prod_{j\ge1}\left(1+\frac{4z}{j_{\theta,j}^2}\right),
\]
where $j_{\theta,j}$ denotes the $j$th positive zero of $J_\theta$.
The positive numbers $4/j_{\theta,j}^2$ are summable.
Consequently $(a_n)_{n\ge0}$ is a P\'olya-frequency sequence of
infinite order by the Schoenberg--Edrei criterion (as used in
\cite{WangYang2018}).  Hence its Toeplitz matrix
$[a_{n-k}]_{n,k\ge0}$ is totally positive.  Positive diagonal
scaling proves the triangle assertion.

For the square we likewise have
\[
 \mathcal N_\theta(i+j,j)
 =\frac{w_{i+j}}{w_iw_j},
\]
so its matrix is positively diagonally equivalent to the Hankel
matrix $[w_{i+j}]_{i,j\ge0}$.  Let $X$ and $Y$ be independent gamma
random variables of respective shapes $1$ and $\theta+1$, both
with unit scale.  Then
\[
 \mathbb E[(XY)^n]
 =\mathbb E[X^n]\,\mathbb E[Y^n]
 =n!(\theta+1)_n=w_n.
\]
Thus $(w_n)$ is a Stieltjes moment sequence whose representing
measure has infinite support in $(0,\infty)$.  Its Hankel matrix
is strictly totally positive (for example, by the integral formula
for arbitrary minors using two generalized Vandermonde
determinants).  Positive diagonal scaling proves the square
assertion.
\end{proof}

\begin{remark}
At $\theta=0$ and $\theta=1$ we recover, respectively, the type~$B$
triangle $\binom nk^2$ and the shifted type~$A$ triangle
$N(n+1,k+1)$.  The usual non-shifted type~$A$ triangle is obtained
by replacing $n$ with $n-1$ in the latter for $n\ge1$:
\[
 \mathcal N_1(n-1,k)
 =\frac1n\binom nk\binom n{k+1}.
\]
Thus the shift is merely an indexing convention, not a different
total-positivity statement.  Notice also that neither matrix theorem
above implies the coefficientwise Hankel total positivity of the
row-polynomial sequence in Theorem~\ref{ar:cor:classical}.
\end{remark}

\subsection{Real-rootedness and strong \texorpdfstring{$t$}{t}-log-convexity}
\label{sec:interpolation}
The coefficient formula also gives
\[
 P_n^{(\theta)}(t)={}_2F_1\!\left(
 \begin{matrix}-n,-n-\theta\\\theta+1\end{matrix};t\right).
\]
\begin{corollary}\label{cor:theta_realrooted}
For every $\theta\ge0$ and $n\ge1$, $P_n^{(\theta)}(t)$
has only negative real zeros.
\end{corollary}
\begin{proof}
Pfaff's transformation and the hypergeometric representation of the
Jacobi polynomials give
\begin{equation}\label{eq:theta_jacobi}
P_n^{(\theta)}(t)=\frac{n!}{(\theta+1)_n}(1-t)^n
\mathrm P_n^{(\theta,\theta)}\!\left(\frac{1+t}{1-t}\right).
\end{equation}
The zeros of $\mathrm P_n^{(\theta,\theta)}$ lie in $(-1,1)$,
and $t=(x-1)/(x+1)<0$ for such $x$.
\end{proof}
For integral $\theta$, Mao and Wang~\cite{MaoWang2026} proved the
stronger preservation of real-rootedness by the associated basis
transformation.
\begin{proof}[Proof of Theorem~\ref{thm:theta_structure_intro}]
Theorem~\ref{ar:cor:classical} establishes Hankel total positivity,
\eqref{ar:eq:classical} gives the $\gamma$-expansion, and
\eqref{eq:classical_affine_intro} supplies the Stieltjes representation.
\end{proof}
\begin{corollary}\label{prop:interpolation_C1}
The sequence $\{P_n^{(\theta)}(t)\}_{n\ge0}$ is strongly
$t$-log-convex for all $\theta\ge0$.
\end{corollary}
\begin{proof}
The defining differences are order-two Hankel minors; apply
Theorem~\ref{ar:cor:classical}.
\end{proof}

\subsection{Super Narayana polynomials and the Rodelet--Causse--Tevlin conjecture}
\label{subsec:super-Narayana-classical}
For integers $s,r\ge0$, write the classical super Narayana
polynomial in the normalization of
Rodelet--Causse and Tevlin~\cite[Definition~2]{RodeletTevlin2025} as
\begin{equation}\label{eq:classical-super-Narayana}
 B_r^{(s)}(t):=\mathcal N_{s+r}^{(s)}(1,t)
 =\frac{\binom{2s}s}{\binom{s+r}s}
 \sum_{k=0}^{r}\binom{s+r}k\binom{s+r}{k+s}t^k.
\end{equation}
Their Conjecture~3 in \cite{RodeletTevlin2025} asks for coefficientwise
$t$-log-convexity of this sequence.

\begin{corollary}[Hankel total positivity of super Narayana polynomials]
\label{cor:super-Narayana-HTP}
For each integer $s\ge0$, the sequence
$\big(B_r^{(s)}(t)\big)_{r\ge0}$ is coefficientwise
Hankel-totally positive. More precisely, if $\ell\ge1$ and
$0\le i_1<\cdots<i_\ell$, $0\le j_1<\cdots<j_\ell$, then
\begin{equation}\label{eq:super-Narayana-Hankel}
 \det\!\left[B_{i_a+j_b}^{(s)}(t)\right]_{a,b=1}^{\ell}
 \in\mathbb R_{\ge0}[t].
\end{equation}
In particular, it is strongly coefficientwise $t$-log-convex,
confirming Conjecture~3 of Rodelet--Causse--Tevlin.
\end{corollary}
\begin{proof}
Cancelling ordinary factorials gives
\[
 \frac{\binom{s+r}k\binom{s+r}{k+s}}{\binom{s+r}s}
 =\binom rk\frac{(s+1)_r}{(s+1)_k(s+1)_{r-k}}.
\]
Comparison with \eqref{ar:eq:classicalcoeff} therefore yields
\begin{equation}\label{eq:super-Narayana-classical-connection}
 B_r^{(s)}(t)=\binom{2s}s P_r^{(s)}(t),
 \qquad B_0^{(s)}(t)=\binom{2s}s.
\end{equation}
Every minor in \eqref{eq:super-Narayana-Hankel} is consequently
$\binom{2s}s^{\ell}$ times the corresponding minor of the Hankel
matrix of $\big(P_r^{(s)}(t)\big)_{r\ge0}$.
Theorem~\ref{ar:cor:classical} proves its coefficientwise
nonnegativity. Strong $t$-log-convexity follows from the
order-two minors; for example, rows $0,1$ and columns $a-1,b$
give $B_{a-1}^{(s)}B_{b+1}^{(s)}-B_a^{(s)}B_b^{(s)}$
for $1\le a\le b$. Taking $a=b$ gives the conjectured
coefficientwise $t$-log-convexity.
\end{proof}

\section{Extreme coefficients of arbitrary Hankel minors}\label{ar:sec:extreme}
The production-matrix factorization gives more than nonnegativity.
Its first nonzero term can be calculated for every choice of indices.

\begin{theorem}[Universal factor and extreme coefficients]\label{ar:thm:extreme}
Let $r_h>0$ and assume that the matrix $J_r(t)$ in \eqref{ar:eq:Jb}
is coefficientwise totally positive. Write $A_n(t)=(J_r(t)^n)_{00}$.
For increasing nonnegative index sets
$I=(i_1<\cdots<i_m)$ and $J=(j_1<\cdots<j_m)$, where $m\ge1$, put
\[
 D_{I,J}(t)=\det[A_{i_a+j_b}(t)]_{a,b=1}^m,\qquad
 N=\sum_{a=1}^m i_a+\sum_{b=1}^m j_b,\qquad c=\binom m2.
\]
Then
\begin{equation}\label{ar:eq:minor-factor}
 D_{I,J}(t)=t^c F_{I,J}(t),\qquad F_{I,J}(t)\in\R_{\ge0}[t],
 \qquad F_{I,J}(t)=t^{N-2c}F_{I,J}(1/t).
\end{equation}
Moreover,
\begin{equation}\label{ar:eq:extreme-coefficient}
 F_{I,J}(0)=[t^{N-2c}]F_{I,J}(t)
 =\left(\prod_{h=1}^{m-1}r_h^{m-h}\right)
 \frac{\displaystyle\prod_{a<b}(i_b-i_a)(j_b-j_a)}
 {\displaystyle\prod_{h=0}^{m-1}(h!)^2}>0.
\end{equation}
Thus the order of vanishing of $D_{I,J}$ at $t=0$ is exactly $c$,
and its degree is exactly $N-c$. The result applies, in particular, to the Gegenbauer weights
$r_h=\rho_h(\theta)$ and to their associated tails.
\end{theorem}
\begin{proof}
Set $O_{nk}=(J_r(t)^n)_{0k}$ and $d_k=t^k\prod_{h=1}^k r_h$.
By Proposition~\ref{ar:prop:transfer} and \eqref{ar:eq:factorH},
Cauchy--Binet gives the finite expansion
\begin{equation}\label{ar:eq:minor-CB}
 D_{I,J}(t)=
 \sum_{0\le k_1<\cdots<k_m}
 \left(\prod_{a=1}^m d_{k_a}\right)
 \det O_{I,K}(t)\det O_{J,K}(t).
\end{equation}
Here $K=(k_1,\ldots,k_m)$ and terms with
$k_m>\min(i_m,j_m)$ vanish. Each term has nonnegative coefficients,
and $\sum_a k_a\ge c$, proving divisibility and nonnegativity of
the quotient.

At $t=0$, the production matrix has diagonal and superdiagonal
entries equal to $1$ and all other entries zero. Hence
$O_{nk}(0)=\binom nk$. Only $K=(0,1,\ldots,m-1)$ contributes to
$[t^c]D_{I,J}$. The elementary determinant identity
\[
 \det\left[\binom{i_a}{b-1}\right]_{a,b=1}^m
 =\frac{\prod_{a<b}(i_b-i_a)}{\prod_{h=0}^{m-1}h!}
\]
follows by expressing the binomial polynomials in the monomial
basis: their degrees are $b-1$ and their leading coefficients are
$1/(b-1)!$. Applying it to $I$ and $J$ proves the constant term in
\eqref{ar:eq:extreme-coefficient}, which is strictly positive.

Finally a length-$n$ Motzkin excursion with $j$ down steps contributes
a constant multiple of $t^j(1+t)^{n-2j}$. Consequently
$A_n(t)=t^n A_n(1/t)$. Factoring $t^{i_a}$ from each row and
$t^{j_b}$ from each column gives $D_{I,J}(t)=t^N D_{I,J}(1/t)$.
After division by $t^c$ this is \eqref{ar:eq:minor-factor}.
The nonzero constant term then determines the exact degree and the
equal highest-degree coefficient.
\end{proof}

\begin{remark}[Palindromicity does not imply gamma positivity]
\label{ar:rem:minors-not-gamma}
Even for the classical Narayana endpoints the quotients $F_{I,J}$
need not be gamma-positive. For $I=(0,1)$ and $J=(1,2)$,
\[
 D_{I,J}(t)=A_1A_3-A_2^2
 =tr_1\bigl((1+t)^2-tr_1\bigr).
\]
For type A ($r_1=1$), its quotient by $t$ is
$1+t+t^2=(1+t)^2-t$, with gamma coefficient $-1$.
For type B ($r_1=2$), the quotient is
$2(1+t^2)=2(1+t)^2-4t$, with gamma coefficient $-4$.
Thus Theorem~\ref{ar:thm:extreme} supplies palindromic polynomials with
nonnegative ordinary coefficients, but not gamma positivity of all
normalized minors.
\end{remark}

\subsection{Connection with Han's convolution powers}
Han~\cite{Han2026} determined Hankel determinants for convolution powers
of the classical Narayana generating function.  Put
\[
 G_\theta(t,z)=\sum_{n\ge0}P_n^{(\theta)}(t)z^n,
 \qquad
 P_n^{(\theta)}(t)=\sum_{k=0}^n
 \binom nk\frac{\binom{n+\theta}{k}}{\binom{k+\theta}{k}}t^k.
\]
If $C(t,z)$ denotes the generating function for the ordinary type~$A$
Narayana polynomials, with constant coefficient $1$, then
\[
 G_1(t,z)=\frac{C(t,z)-1}{z}.
\]
Thus Han's even-indexed convolution powers arise at $\theta=1$ in the family
\[
 H_N^{(\theta,m)}(t)=
 \det\bigl([z^{i+j}]G_\theta(t,z)^m\bigr)_{0\le i,j<N}.
\]
The coefficientwise Hankel-total-positivity results proved here concern
$m=1$ and do not extend automatically to $m>1$.  For example,
\[
 G_1(t,z)=1+(1+t)z+(1+3t+t^2)z^2+O(z^3),
 \qquad H_2^{(1,2)}(t)=-(1+t^2).
\]
The type~$B$ endpoint $\theta=0$ admits a separate, particularly
simple analysis.  Write $C_n^{(\lambda)}(x)$ for the Gegenbauer
polynomial, characterized by
$\sum_{n\ge0}C_n^{(\lambda)}(x)z^n=(1-2xz+z^2)^{-\lambda}$.

\begin{proposition}[Type~$B$ convolution powers and finite Hankel rank]
\label{prop:Han-typeB-convolution}
For every positive integer $m$,
\begin{equation}\label{eq:Han-typeB-generating}
 G_0(t,z)^m
 =\bigl(1-2(1+t)z+(1-t)^2z^2\bigr)^{-m/2},
\end{equation}
and hence
\begin{equation}\label{eq:Han-typeB-Gegenbauer}
 [z^n]G_0(t,z)^m
 =(1-t)^n C_n^{(m/2)}\!\left(\frac{1+t}{1-t}\right).
\end{equation}
The right-hand side is interpreted as a polynomial in $t$ (including
$t=1$).  If $m=2r$ is even, then
\begin{equation}\label{eq:Han-typeB-rank}
 H_N^{(0,2r)}(t)=0\qquad(N>2r).
\end{equation}
Moreover, for every $m\ge1$,
\begin{equation}\label{eq:Han-typeB-second}
 H_2^{(0,m)}(t)
 =\frac m2\bigl((2-m)(1+t)^2-(1-t)^2\bigr).
\end{equation}
\end{proposition}
\begin{proof}
At $\theta=0$, the row polynomials are
$P_n^{(0)}(t)=\sum_{k=0}^n\binom nk^2t^k$.  Their standard
Legendre-polynomial generating function gives
\[
 G_0(t,z)=\bigl(1-2(1+t)z+(1-t)^2z^2\bigr)^{-1/2}.
\]
Raising this identity to the $m$th power gives
\eqref{eq:Han-typeB-generating}; the Gegenbauer generating function,
with $z$ replaced by $(1-t)z$, gives
\eqref{eq:Han-typeB-Gegenbauer}.  For $m=2r$,
$G_0(t,z)^{2r}$ is rational with denominator of degree $2r$,
and its coefficients satisfy a constant-coefficient recurrence
of order at most $2r$.  Every Hankel matrix of size $N>2r$
is therefore singular.  Finally, writing $a=1+t$ and
$b=(1-t)^2$, the coefficients of orders $0,1,2$ in
$(1-2az+bz^2)^{-m/2}$ are
\[
 1,\qquad ma,\qquad \frac{m(m+2)}2a^2-\frac m2b.
\]
Their $2\times2$ Hankel determinant is
$\frac m2((2-m)a^2-b)$, which proves
\eqref{eq:Han-typeB-second}.
\end{proof}

In particular,
\[
 H_2^{(0,2)}(t)=-(1-t)^2,
 \qquad
 H_2^{(0,3)}(t)=-3(1+t^2),
\]
whereas $H_2^{(1,2)}(t)=-(1+t^2)$ as above.
For $m=1$, the arcsine Jacobi coefficients also give
\[
 H_N^{(0,1)}(t)=2^{N-1}t^{\binom N2}\qquad(N\ge1).
\]
For odd convolution powers $m\ge3$, it would be interesting to
seek explicit Hankel determinants analogous to Han's type~$A$
evaluations; the rational finite-rank argument no longer applies.
More generally, it remains open to what extent Han's formulas
admit analogues throughout the continuous parameter $\theta$.

\section{Preliminaries and the Zhu--Sun criterion}\label{sec:prelim} 

For polynomials \(f(t),g(t)\in\R[t]\), we write
\[
f(t)\qge g(t)
\]
if \(f(t)-g(t)\) has nonnegative coefficients; this is the same order
as $\coef$ used above.

\begin{definition}
A sequence \(\{f_n(t)\}_{n\geq0}\) of polynomials with nonnegative
coefficients is \emph{strongly \(t\)-log-convex} if
\begin{equation}
f_{n-1}(t)f_{m+1}(t)-f_n(t)f_m(t)\qge0
\qquad(m\geq n\geq1).
\label{eq:strong}
\end{equation}
\end{definition}

Let \(a(n,k)\) be a triangle of nonnegative real numbers, extended by zero
for $k<0$ or $k>n$, and set
\[
A_n(t)=\sum_{k=0}^n a(n,k)t^k.
\]
For \(m\geq n>0\), \(0\leq h\leq m+n\), and
\(0\leq k\leq\lfloor h/2\rfloor\), Zhu and Sun define
\begin{equation}
\alpha_k(m,n,h)=
\begin{cases}
a(n-1,k)a(m+1,h-k)
+a(m+1,k)a(n-1,h-k)
\\ \qquad
-a(n,k)a(m,h-k)
-a(m,k)a(n,h-k),
& k<h/2,\\[2mm]
a(n-1,k)a(m+1,k)-a(n,k)a(m,k),
& k=h/2.
\end{cases}
\label{eq:alpha}
\end{equation}

Zhu and Sun~\cite[Theorem~2.1]{ZhuSun2015} assume an index \(r\)
such that \(\alpha_k\geq0\) for \(k\leq r\) and \(\alpha_k<0\)
for \(k>r\).  For the present application, mixed Hoggatt minors
can have trailing zeros outside the support of the smaller row.
We therefore use a weak-sign variant with \(\alpha_k\leq0\)
after the change, and include its short proof for clarity.

\begin{lemma}[Weak-sign Zhu--Sun criterion]
\label{lem:ZSweak}
Suppose that
\begin{enumerate}[label=\textup{(C\arabic*)}]
\item \(\{A_n(t)\}_{n\geq0}\) is strongly \(t\)-log-convex;
\item for every \(m\geq n>0\) and \(0\leq h\leq m+n\), there is an index
\(r=r(m,n,h)\) with $0\le r\le\lfloor h/2\rfloor$ such that
\[
\alpha_k(m,n,h)\geq0\quad(0\leq k\leq r),
\qquad
\alpha_k(m,n,h)\leq0\quad(r<k\leq\lfloor h/2\rfloor).
\]
\end{enumerate}
Then the transformation
\[
B_n(t)=\sum_{k=0}^n a(n,k)f_k(t)
\]
preserves strong \(t\)-log-convexity for input sequences with
nonnegative polynomial coefficients.
\end{lemma}

\begin{proof}
This is the proof of Zhu and Sun with non-strict inequalities.  For fixed
\(m\geq n>0\) and \(h\), put
\[
p_k(t)=f_k(t)f_{h-k}(t),
\qquad
A(m,n,h)=\sum_{k=0}^{\lfloor h/2\rfloor}\alpha_k(m,n,h).
\]
Strong \(t\)-log-convexity of \(\{f_k(t)\}\) gives
\[
p_0(t)\qge p_1(t)\qge\cdots\qge
p_{\lfloor h/2\rfloor}(t).
\]
For $k\le r$, both $\alpha_k$ and $p_k-p_r$ are coefficientwise
nonnegative. For $k>r$, both $-\alpha_k$ and $p_r-p_k$ are
coefficientwise nonnegative. Therefore
\[
 \sum_k\alpha_k(m,n,h)p_k(t)-p_r(t)\sum_k\alpha_k(m,n,h)
 =\sum_k\alpha_k(m,n,h)(p_k(t)-p_r(t))\qge0.
\]
Condition \textup{(C1)} says precisely that
\(A(m,n,h)\geq0\), since it is the coefficient of \(t^h\) in
\(A_{n-1}(t)A_{m+1}(t)-A_n(t)A_m(t)\).  Hence the last expression has
nonnegative coefficients.  Summing over \(h\) proves
\[
B_{n-1}(t)B_{m+1}(t)-B_n(t)B_m(t)\qge0.
\]
\end{proof}

For the Hoggatt triangle
\[
a(n,k)=\Hog{n}{k}{d},
\]
condition \textup{(C1)} follows from the strong \(t\)-log-convexity
theorem of Chern and Shi~\cite[Theorem~1.2]{ChernShi2026}.  We therefore concentrate
on condition \textup{(C2)}.

We use the rising-factorial notation
\[
R_d(x):=(x)_d=x(x+1)\cdots(x+d-1).
\]
When \(d\) is fixed, we abbreviate \(R_d(x)\) to \(R(x)\).
A direct calculation from \eqref{eq:hoggatt} gives
\begin{equation}
\frac{a(N+1,j)}{a(N,j)}
=
\frac{R(N+1)}{R(N-j+1)},
\qquad
\frac{a(N-1,j)}{a(N,j)}
=
\frac{R(N-j)}{R(N)}.
\label{eq:hoggatt-ratios}
\end{equation}

\section{Hoggatt transformations}\label{sec:hoggattproof}
\subsection{A hypergeometric hitting inequality}\label{sec:hypergeom_ineq}

For integers $x\ge0$ and $r\ge1$, define
\begin{equation}\label{eq:def_Fd}
F_{d}(x, r)=1-\frac{R_{d}(x)}{R_{d}(x+r)}.
\end{equation}
Set $F_d(x,0)=0$ separately for every $x\ge0$; this convention
also defines $F_d(0,0)$ without a quotient $0/0$.

The following local inequality is our main technical building block.
It has a direct combinatorial interpretation in terms of hitting probabilities for hypergeometric sampling.

\begin{lemma}[Hypergeometric hitting inequality]\label{lem:hypergeom_hit}
For integers $d\ge 1$, $x\ge 0$, and $r\ge 2$,
\begin{equation}\label{eq:local}
F_{d}(x+1, r)\,F_{d}(x+1, r-2)
\ge F_{d}(x, r)\,F_{d}(x+2, r-2).
\end{equation}
\end{lemma}

\begin{proof}
Set $M:=x+d-1$, $R:=r$.
Using the identity $(x)_d = d!\dbinom{x+d-1}{d}$, we rewrite
\begin{equation}\label{eq:Fd_hyper_prob}
F_{d}(x,r)=1-\frac{\dbinom{M}{d}}{\dbinom{M+R}{d}}.
\end{equation}
Consider a population containing $M$ unmarked objects and $R$ marked objects.
Draw a $d$-element subset uniformly at random.
Let $J$ denote the number of marked objects selected.
Then
\[
F_{d}(x,r)=\mathbb{P}_{M,R}(J\ge 1),
\]
the probability that the random $d$-subset hits at least one marked item.
Write $p(M,R):=\mathbb{P}_{M,R}(J\ge 1)$.

If $R=2$, then both sides of~\eqref{eq:local} vanish identically, so the inequality holds trivially.
We henceforth assume $R\ge 3$.
The four factors in~\eqref{eq:local} correspond to parameter pairs
\[
(M+1,R),\quad (M+1,R-2),\quad (M,R),\quad (M+2,R-2).
\]
Inequality~\eqref{eq:local} is equivalent to
\begin{equation}\label{eq:theta_ineq}
\Theta(M,R)\ge \Theta(M+1,R-2),
\qquad \text{where }
\Theta(M,R):=\frac{p(M+1,R)}{p(M,R)}.
\end{equation}

Define the unnormalized hypergeometric weights
\[
w_{j}(M,R)=\dbinom{R}{j}\dbinom{M}{d-j},\qquad j\ge 0,
\]
and let
\[
S_{d}(M,R)=\sum_{j\ge 1} w_{j}(M,R).
\]
Then $S_d(M,R)$ counts favourable $d$-subsets intersecting the marked block.
We set $S_0(M,R)=0$. Since $M\ge d-1$ and $R\ge3$, both
$S_d(M,R)$ and $w_1(M,R)=R\binom M{d-1}$ are positive; the
same holds for $(M+1,R-2)$. Thus all conditional probabilities and
relative odds used below are well defined, also when $x=0$ or $d=1$.
Introduce the conditional probability
\[
\pi_1(M,R)=\frac{w_{1}(M,R)}{S_{d}(M,R)}
=\mathbb{P}_{M,R}\bigl(J=1\mid J\ge 1\bigr).
\]
Let $N=M+R$ be total population size.
We claim
\begin{equation}\label{eq:theta_formula}
\Theta(M,R)=1-\frac{\pi_1(M,R)}{N+1}.
\end{equation}

To establish~\eqref{eq:theta_formula}, observe that adding one extra unmarked object gives
\[
S_{d}(M+1,R)=S_{d}(M,R)+S_{d-1}(M,R).
\]
Double-count pairs $(A,e)$ where $A$ is a $d$-subset with at least one marked element, $e\in A$,
and $A\setminus\{e\}$ still contains a marked element. This yields
\[
(N-d+1)S_{d-1}(M,R)=d\,S_{d}(M,R)-w_{1}(M,R),
\]
and rearranging,
\[
\frac{S_{d-1}(M,R)}{S_{d}(M,R)}
=\frac{d-\pi_1(M,R)}{N-d+1}.
\]
Combined with $\dfrac{\dbinom{N}{d}}{\dbinom{N+1}{d}}=\dfrac{N-d+1}{N+1}$, one obtains~\eqref{eq:theta_formula}.

We now compare conditional probabilities $\pi_1(M,R)$ and $\pi_1(M+1,R-2)$.
The positive supports for $j\ge1$ are respectively
$\{1,\ldots,\min(d,R)\}$ and
$\{1,\ldots,\min(d,R-2)\}$.
On their common support for $j\ge 2$, compute the ratio of weights:
\[
\frac{w_{j}(M,R)/w_{1}(M,R)}{w_{j}(M+1,R-2)/w_{1}(M+1,R-2)}
=\frac{(R-1)(R-2)}{(R-j)(R-j-1)}\cdot \frac{M-d+j+1}{M-d+2}\ge 1.
\]
For $d=1$, both conditional probabilities equal $1$. In general,
terms with $j>R-2$ appear only for the distribution with parameters $(M,R)$,
adding further probability mass away from the event $J=1$.
Consequently
\begin{equation}\label{eq:q_ineq}
\pi_1(M,R)\le \pi_1(M+1,R-2).
\end{equation}
In other words: conditioning on hitting the marked set, the law for $(M,R)$ is stochastically shifted toward larger values of $J$.

The total population sizes satisfy $N=M+R$ for $(M,R)$, and $N-1$ for $(M+1,R-2)$.
Substitute~\eqref{eq:q_ineq} into~\eqref{eq:theta_formula}:
\[
\Theta(M,R)=1-\frac{\pi_1(M,R)}{N+1}
\ge 1-\frac{\pi_1(M+1,R-2)}{N}
=\Theta(M+1,R-2).
\]
This proves~\eqref{eq:theta_ineq}, and hence the original inequality~\eqref{eq:local}.
\end{proof}

\begin{remark}[$\mathrm{TP}_2$ interpretation]\label{rem:TP2_interpret}
The preceding proof compares two hypergeometric laws by likelihood ratios.
Place the weights in the two-row array
\[
\begin{pmatrix}
 (w_j(M+1,R-2))_{j\ge1}\\
 (w_j(M,R))_{j\ge1}
\end{pmatrix}.
\]
The likelihood ratio
\[
\frac{w_{j}(M,R)}{w_{j}(M+1,R-2)}
=\frac{R(R-1)}{(R-j)(R-j-1)}\cdot\frac{M-d+j+1}{M+1}
\]
is increasing in $j$ over the common support.
Thus the pair forms a $\mathrm{TP}_2$ array.
Together with~\eqref{eq:theta_formula}, this comparison proves
Lemma~\ref{lem:hypergeom_hit}.
For background on monotone likelihood ratio and total positivity we refer to Karlin~\cite{Karlin1968}.
\end{remark}

The local inequality telescopes to give the following corollary.

\begin{corollary}\label{cor:Phi} 
Let $d\ge 1$, $0\le a<b$, $c\ge 1$ be integers with $b-a\ge 3$. Define
\begin{equation}\label{eq:def_Phi}
\Phi_{d}(a,b;c)=\frac{1-R(a+c)/R(b+c)}{1-R(a)/R(b)}.
\end{equation}
Then
\[
\Phi_{d}(a,b;c)\ge \Phi_{d}(a+1,b-1;c).
\]
\end{corollary}

\begin{proof}
Set $r:=b-a$. From definition~\eqref{eq:def_Fd},
\[
\Phi_{d}(a,b;c)=\frac{F_{d}(a+c,r)}{F_{d}(a,r)}.
\]
Lemma~\ref{lem:hypergeom_hit} is equivalent to
\[
\frac{F_{d}(x+1,r)}{F_{d}(x,r)}
\ge \frac{F_{d}(x+2,r-2)}{F_{d}(x+1,r-2)}.
\]
Multiply this inequality successively for $x=a,a+1,\dots,a+c-1$:
\[
\frac{F_{d}(a+c,r)}{F_{d}(a,r)}
\ge \frac{F_{d}(a+c+1,r-2)}{F_{d}(a+1,r-2)},
\]
which is exactly $\Phi_{d}(a,b;c)\ge \Phi_{d}(a+1,b-1;c)$.
\end{proof}

\subsection{The mixed one-sign-change property}\label{sec:sign_change}

We now prove condition \textup{(C2)} of Lemma~\ref{lem:ZSweak}.

\begin{proposition}
\label{prop:C2}
For every \(d\geq1\), the Hoggatt triangle
\[
a(n,k)=\Hog{n}{k}{d}
\]
satisfies condition \textup{(C2)} of Zhu and Sun.
\end{proposition}

\begin{proof}
Fix \(m\geq n>0\) and \(0\leq h\leq m+n\).  Put
\[
l=h-k
\]
for \(0\leq k<h/2\).

We first dispose of the support and boundary ranges.

If \(k>n\), then \(l\geq k>n\), so every term in
\(\alpha_k(m,n,h)\) containing the \(n\)-th or \((n-1)\)-st row
vanishes; hence
\[
\alpha_k(m,n,h)=0.
\]
These are trailing zeros and are harmless for the weak-sign condition.

We may therefore assume \(k\leq n\).  If \(l>m+1\), then
\(\alpha_k(m,n,h)=0\).  If \(l=m+1\), then necessarily \(k\leq n-1\)
(because \(h\leq m+n\)), and
\[
\alpha_k(m,n,h)=a(n-1,k)a(m+1,m+1)>0.
\]
If \(n<l\leq m\), then
\[
a(n,l)=a(n-1,l)=0,
\]
and therefore
\begin{equation}
\alpha_k(m,n,h)
=
a(n,k)a(m,l)\,(P_k-1),
\label{eq:boundary-P}
\end{equation}
where
\begin{equation}
P_k=
\frac{R(n-k)R(m+1)}
     {R(n)R(m-l+1)}.
\label{eq:Pk-boundary}
\end{equation}
As \(k\) increases, \(n-k\) decreases while \(m-l+1\) increases, so
\(P_k\) strictly decreases.  Thus the signs in this range can change
at most once from nonnegative to nonpositive.

It remains to consider \(l\leq n\). Since \(k<l\), we have
\(k<n\); the two entries $a(n,k)$ and $a(m,l)$ used for
normalization are strictly positive. The factor $R(n-l)$ may be zero
and will be handled explicitly. Set
\begin{equation}
a=n-l,\qquad
b=n-k,\qquad
c=m-n+1,
\qquad
\lambda=\frac{R(n+c)}{R(n)}.
\label{eq:abcd}
\end{equation}
Then
\[
0\leq a<b\leq n,\qquad c\geq1,
\]
where the case $m=n$ gives $c=1$. Products over $1,\ldots,c-1$ are then empty and equal to $1$. The replacement \(k\mapsto k+1\) sends
\[
(a,b)\longmapsto(a+1,b-1).
\]

Using \eqref{eq:hoggatt-ratios}, divide
\(\alpha_k(m,n,h)\) by \(a(n,k)a(m,l)\).  We obtain
\begin{equation}
\frac{\alpha_k(m,n,h)}
     {a(n,k)a(m,l)}
=
(P_k-1)-\rho_k(1-Q_k),
\label{eq:normalized-alpha}
\end{equation}
where
\begin{align}
P_k&=
\lambda\frac{R(b)}{R(a+c)},
\label{eq:Pk}\\
Q_k&=
\lambda\frac{R(a)}{R(b+c)},
\label{eq:Qk}\\
\rho_k&=
\prod_{s=1}^{c-1}
\frac{R(a+s)}{R(b+s)}.
\label{eq:rhok}
\end{align}

The function
\[
x\longmapsto \frac{R(x+c)}{R(x)}
=
\prod_{j=0}^{d-1}
\left(1+\frac{c}{x+j}\right)
\]
is decreasing for \(x>0\), with the evident limiting interpretation at
\(x=0\).  If \(a=0\), then \(Q_k=0\); otherwise, since \(a<b\leq n\),
the monotonicity of the displayed ratio gives
\[
Q_k<1.
\]
Moreover \(P_k\) strictly decreases under
\((a,b)\mapsto(a+1,b-1)\).  Hence, once \(P_k\leq1\),
\eqref{eq:normalized-alpha} gives \(\alpha_k<0\), and every subsequent
interior term is strictly negative.

We may therefore assume \(P_k>1\) and define
\begin{equation}
T_k=
\frac{P_k-1}{\rho_k(1-Q_k)}.
\label{eq:Tk}
\end{equation}
Then
\begin{equation}
\alpha_k(m,n,h)\geq0
\quad\Longleftrightarrow\quad
T_k\geq1.
\label{eq:alpha-T}
\end{equation}
We compare consecutive noncentral indices, so $k+1<h/2$ and
$b-a=h-2k\ge3$. We prove that \(T_k\) decreases whenever the
next numerator remains positive.

From \eqref{eq:Pk}--\eqref{eq:rhok},
\begin{equation}
T_k=K(a,b)\,H_\lambda(a,b),
\label{eq:factor-T}
\end{equation}
where
\begin{align}
K(a,b)
&=
\prod_{s=1}^{c}
\frac{R(b+s)}{R(a+s)},
\label{eq:Kab}\\
H_\lambda(a,b)
&=
\frac{\lambda R(b)-R(a+c)}
     {R(b+c)-\lambda R(a)}.
\label{eq:Hlambda}
\end{align}
Clearly
\begin{equation}
K(a,b)>K(a+1,b-1).
\label{eq:Kmon}
\end{equation}

If
\[
\lambda R(b-1)\leq R(a+c+1),
\]
then \(P_{k+1}\leq1\), and all subsequent mixed coefficients are
nonpositive.  We may therefore suppose
\begin{equation}
\lambda R(b-1)>R(a+c+1).
\label{eq:next-positive}
\end{equation}

Set
\begin{equation}
L=\frac{R(a+c+1)}{R(b-1)},
\qquad
U=\frac{R(b+c)}{R(b)}.
\label{eq:LU}
\end{equation}
By \eqref{eq:next-positive}, \(L<\lambda\).  Since \(b\leq n\) and
\(R(x+c)/R(x)\) decreases in \(x\),
\[
\lambda=\frac{R(n+c)}{R(n)}
\leq
\frac{R(b+c)}{R(b)}
=U.
\]
Thus
\begin{equation}
L<\lambda\leq U.
\label{eq:lambda-interval}
\end{equation}

Consider
\begin{align}
E(\mu)
={}&
[\mu R(b)-R(a+c)]
[R(b+c-1)-\mu R(a+1)]
\nonumber\\
&-
[\mu R(b-1)-R(a+c+1)]
[R(b+c)-\mu R(a)].
\label{eq:E}
\end{align}
The inequality
\[
H_\mu(a,b)\geq H_\mu(a+1,b-1)
\]
is equivalent to $E(\mu)\ge0$. To check the denominators explicitly, put
$U'=R(b+c-1)/R(b-1)$. The decreasing ratio $R(x+c)/R(x)$ gives
$U<U'$. For $\mu\le U$,
\begin{align*}
 R(b+c)-\mu R(a)
 &\ge R(b+c)\bigl(1-R(a)/R(b)\bigr)>0,\\
 R(b+c-1)-\mu R(a+1)
 &\ge R(b+c-1)\bigl(1-R(a+1)/R(b-1)\bigr)>0.
\end{align*}
The second bound uses $b-1>a+1$. Thus cross-multiplication is valid
throughout $[L,U]$.

The coefficient of \(\mu^2\) in \(E(\mu)\) is
\[
R(b-1)R(a)-R(b)R(a+1)<0.
\]
Hence \(E\) is concave.

At \(\mu=L\), the second numerator in \eqref{eq:E} vanishes.  The
first numerator is positive since \(R(b)>R(b-1)\) and
\(R(a+c+1)>R(a+c)\).  The remaining denominator is also positive:
\(b-1>a+1\) implies
\(L R(a+1)<R(a+c+1)<R(b+c-1)\).  Consequently
\begin{equation}
E(L)\geq0.
\label{eq:EL}
\end{equation}

At \(\mu=U\),
\begin{equation}
H_U(a,b)
=
\frac{1-R(a+c)/R(b+c)}
     {1-R(a)/R(b)}
=
\Phi_d(a,b;c).
\label{eq:HU}
\end{equation}
Recall that $U'=R(b+c-1)/R(b-1)>U$.
Furthermore \(H_\mu(a+1,b-1)\) is increasing in \(\mu\), because
\[
\frac{\partial}{\partial\mu}H_\mu(a+1,b-1)
=
\frac{
R(b-1)R(b+c-1)-R(a+1)R(a+c+1)}
{[R(b+c-1)-\mu R(a+1)]^2}
>0.
\]
Therefore, by Corollary~\ref{cor:Phi},
\begin{align*}
H_U(a+1,b-1)
&\leq H_{U'}(a+1,b-1)\\
&=\Phi_d(a+1,b-1;c)\\
&\leq\Phi_d(a,b;c)
=H_U(a,b).
\end{align*}
Hence
\begin{equation}
E(U)\geq0.
\label{eq:EU}
\end{equation}

Since \(E\) is concave and \(\lambda\in[L,U]\),
\[
E(\lambda)\geq0.
\]
Thus
\begin{equation}
H_\lambda(a,b)
\geq
H_\lambda(a+1,b-1).
\label{eq:Hmon}
\end{equation}
Combining \eqref{eq:Kmon} and \eqref{eq:Hmon} gives
\[
T_k\geq T_{k+1}.
\]
By \eqref{eq:alpha-T}, the noncentral coefficients within the support
can therefore change sign at most once, from nonnegative to nonpositive.

It remains only to consider the central term when \(h\) is even and
\(k=h/2\).  If \(k>n\), then the central coefficient is \(0\).  If
\(1\leq k\leq n\), put \(b=n-k<n\).  Then
\[
\frac{a(n-1,k)a(m+1,k)}
     {a(n,k)a(m,k)}
=
\frac{R(n+c)}{R(n)}
\frac{R(b)}{R(b+c)}
<1,
\]
Indeed, the ratio is zero when \(b=0\); when \(b>0\), it is
strictly less than one because \(R(x+c)/R(x)\) decreases in \(x\).
Hence the central coefficient is negative.  For \(h=k=0\), it is \(0\).

The boundary and interior arguments fit together as follows. If a
boundary coefficient with $n<l\le m$ is negative, then $P_k<1$;
if it is zero, then $P_k=1$. The same expression $P_k$ decreases
through the transition to $l\le n$, where
$(P_k-1)-\rho_k(1-Q_k)<0$ whenever $P_k\le1$.
Hence a nonpositive boundary term cannot be followed by a positive
interior term. Leading zero terms with $l>m+1$ and the positive term
with $l=m+1$ are consistent with the same sign pattern.  Any terms with \(k>n\)
are trailing zeros.  Consequently, for some index \(r\),
\[
\alpha_k(m,n,h)\geq0\quad(k\leq r),
\qquad
\alpha_k(m,n,h)\leq0\quad(k>r).
\]
This is condition \textup{(C2)} of Lemma~\ref{lem:ZSweak}.
\end{proof}

\subsection{Completion of the Hoggatt preservation theorem}\label{sec:proof_main} 

\begin{proof}[Proof of Theorem~\ref{thm:main}]
Take
\[
a(n,k)=\Hog{n}{k}{d}.
\]
By Chern and Shi~\cite[Theorem~1.2]{ChernShi2026}, the row-generating polynomials
\[
A_n(t)=\sum_{k=0}^n a(n,k)t^k
\]
are strongly \(t\)-log-convex.  Hence condition \textup{(C1)} of
Lemma~\ref{lem:ZSweak} holds.

By Proposition~\ref{prop:C2}, the Hoggatt triangle also satisfies
condition \textup{(C2)}.  Lemma~\ref{lem:ZSweak} therefore implies that
for every strongly \(t\)-log-convex input sequence
\(\{f_k(t)\}_{k\geq0}\), the transformed sequence
\[
G_n(t)=\sum_{k=0}^n
\Hog{n}{k}{d}f_k(t)
\]
is strongly \(t\)-log-convex.
\end{proof}

\section{Continuous Narayana transformations}\label{sec:theta_preservation}

Set
\begin{equation}\label{eq:Rtheta}
R_\theta(x)=x(x+\theta),
\qquad
F_\theta(x,r)=1-\frac{R_\theta(x)}{R_\theta(x+r)}
\quad(x\ge0,\ r\ge1).
\end{equation}
Set $F_\theta(x,0)=0$ for $x\ge0$.
The analogue of Lemma~\ref{lem:hypergeom_hit} has an elementary
factorization.

\begin{lemma}\label{lem:theta_local}
For $\theta\ge0$, integers $x\ge0$ and $r\ge2$,
\begin{equation}\label{eq:theta_local}
F_\theta(x+1,r)F_\theta(x+1,r-2)
\ge
F_\theta(x,r)F_\theta(x+2,r-2).
\end{equation}
\end{lemma}

\begin{proof}
A direct simplification gives
\begin{align*}
&F_\theta(x+1,r)F_\theta(x+1,r-2)
-F_\theta(x,r)F_\theta(x+2,r-2)\\
&\quad=
\frac{r(r-2)(r+\theta+2x)(r+\theta+2x+2)}
{(r+x)^2(r+\theta+x)^2(r+x-1)(r+x+1)
 (r+\theta+x-1)(r+\theta+x+1)}\\
&\qquad\quad\times
\bigl(
2r^2+2r\theta+4rx+\theta^2+2\theta x+2x^2-1
\bigr).
\end{align*}
All denominator factors are positive.  The last factor is at least
$2r^2-1>0$, while $r(r-2)\ge0$, proving \eqref{eq:theta_local}.
\end{proof}

Exactly as in Corollary~\ref{cor:Phi}, Lemma~\ref{lem:theta_local}
telescopes. Thus, for integers $0\le a<b$, $c\ge1$ and $b-a\ge3$,
\begin{equation}\label{eq:Phi_theta}
\Phi_\theta(a,b;c)
:=
\frac{1-R_\theta(a+c)/R_\theta(b+c)}
     {1-R_\theta(a)/R_\theta(b)}
\end{equation}
satisfies
\begin{equation}\label{eq:Phi_theta_mon}
\Phi_\theta(a,b;c)
\ge
\Phi_\theta(a+1,b-1;c).
\end{equation}

\begin{proposition}\label{prop:interpolation_C2}
For every real $\theta\ge0$, the triangle
$a_\theta(n,k)=\mathcal N_\theta(n,k)$ satisfies condition
\textup{(C2)} of Lemma~\ref{lem:ZSweak}.
\end{proposition}

\begin{proof}
We give the support and boundary checks explicitly, then verify
that the interior one-sign-change argument of
Proposition~\ref{prop:C2} applies with $R(x)=x(x+1)$ replaced
by $R_\theta(x)=x(x+\theta)$.
Fix $m\ge n>0$ and $0\le h\le m+n$, and put $l=h-k$ for
$0\le k<h/2$.  For $k>n$, the mixed coefficient is zero.
For $k\le n$ and $l>m+1$, it is again zero; if $l=m+1$,
only the term $a_\theta(n-1,k)a_\theta(m+1,m+1)$ survives,
which is positive.  For $n<l\le m$, division by
$a_\theta(n,k)a_\theta(m,l)>0$ gives $P_k-1$ with
\[
P_k=\frac{R_\theta(n-k)R_\theta(m+1)}
{R_\theta(n)R_\theta(m-l+1)}.
\]
This decreases as $k$ increases, so its signs have at most one
change from nonnegative to nonpositive.  These are precisely the
boundary cases needed for the weak-sign condition; zero coefficients
outside the support are allowed.

In the remaining interior range $l\le n$, the variables
$a=n-l$, $b=n-k$, and $c=m-n+1$ satisfy
$0\le a<b\le n$ and $c\ge1$.  From
\eqref{ar:eq:classicalcoeff},
\begin{equation}\label{eq:theta_ratios}
\frac{a_\theta(N+1,j)}{a_\theta(N,j)}
=
\frac{R_\theta(N+1)}{R_\theta(N-j+1)},
\qquad
\frac{a_\theta(N-1,j)}{a_\theta(N,j)}
=
\frac{R_\theta(N-j)}{R_\theta(N)}.
\end{equation}
Moreover, for $c>0$,
\[
x\longmapsto
\frac{R_\theta(x+c)}{R_\theta(x)}
=
\left(1+\frac cx\right)
\left(1+\frac{c}{x+\theta}\right)
\]
is decreasing on $x>0$.  The endpoint comparison used in the proof of
Proposition~\ref{prop:C2} is supplied by
\eqref{eq:Phi_theta_mon}.

For completeness, the concavity step also remains valid: in the notation
of \eqref{eq:E}, the coefficient of $\mu^2$ becomes
\begin{align*}
&R_\theta(b-1)R_\theta(a)
-R_\theta(b)R_\theta(a+1)\\
&\qquad=-(a+b+\theta)
\bigl(2ab+\theta(a+b)+b-a\bigr)<0
\end{align*}
whenever $0\le a<b$ and $b-a\ge3$.  These are exactly the
consecutive noncentral indices for which $k+1<h/2$; no comparison
with a nonexistent central partner is needed.  At $a=0$, the
quantity $Q_k$ of \eqref{eq:Qk} vanishes because
$R_\theta(0)=0$, so the normalized expression and its endpoint
limit remain well-defined.  All the remaining inequalities in
that proof use only the monotonicity just displayed, the fact that
$R_\theta$ is increasing on $[0,\infty)$, and the telescoping inequality
\eqref{eq:Phi_theta_mon}.  Hence the same normalized quantity $T_k$ is
nonincreasing until its numerator becomes nonpositive, and the mixed
coefficients have at most one sign change from nonnegative to
nonpositive.  This is precisely condition \textup{(C2)}.
\end{proof}

\begin{proof}[Proof of Theorem~\ref{thm:interpolation}]
For the triangle $a_\theta(n,k)$, Proposition~\ref{prop:interpolation_C1}
gives condition \textup{(C1)} of Lemma~\ref{lem:ZSweak}, and
Proposition~\ref{prop:interpolation_C2} gives condition \textup{(C2)}.
The weak-sign Zhu--Sun criterion therefore proves that
\[
G_n^{(\theta)}(t)
=
\sum_{k=0}^n a_\theta(n,k)f_k(t)
\]
is strongly $t$-log-convex whenever the input sequence is strongly
$t$-log-convex.
\end{proof}

\begin{remark}\label{rem:interpolation_endpoints}
At $\theta=0$, Theorem~\ref{thm:interpolation} recovers the type~$B$
Narayana transformation
\[
G_n(t)=\sum_{k=0}^n\binom nk^2 f_k(t)
\]
treated by Zhu and Sun~\cite{ZhuSun2015}.  At $\theta=1$, it gives the
shifted type~$A$ Narayana transformation
\[
G_n(t)=\sum_{k=0}^nN(n+1,k+1)f_k(t),
\]
which is exactly the $d=2$ specialization of Theorem~\ref{thm:main}.
Thus Theorem~\ref{thm:interpolation} gives a continuous preservation
theorem joining these two Narayana kernels.
\end{remark}

\begin{remark}[Ordinary versus shifted Narayana transformations]
\label{rem:ZhuSun_index_shift}
For $n\ge1$, write
\[
 N(n,k)=\frac1n\binom nk\binom n{k-1}\qquad(0\le k\le n).
\]
In Proposition~2.11 of Zhu and Sun~\cite{ZhuSun2015}, the ordinary
Narayana transformation is written as
$\mathcal Z_n(f)=\sum_{k=0}^n N(n,k)f_k$.
Although no separate value at $k=0$ is specified there, the usual
convention $\binom n{-1}=0$ yields $N(n,0)=0$ for $n\ge1$;
the formula $\frac1n\binom nk\binom n{k-1}$ does not apply at $n=0$.
Using $N(n+1,k+1)=\mathcal N_1(n,k)$, for
$g_k=f_{k+1}$ we have the exact identity
\[
 \mathcal Z_{n+1}(f)
   =\sum_{k=0}^n\mathcal N_1(n,k)g_k
   =G_n^{(1)}(g)\qquad(n\ge0).
\]
Thus the ordinary positive-indexed Narayana transformation and the
shifted type~$A$ transformation are identical after simultaneous
reindexing of inputs and outputs.  Their preservation statements,
however, are not formally equivalent \emph{as formulated}:
Zhu and Sun assume strong $t$-log-convexity of the full input
$(f_k)_{k\ge0}$, including $f_0$, even though this term does not
contribute to $\mathcal Z_n(f)$ for $n\ge1$.  Our theorem imposes
this condition only on the reindexed input $(g_k)_{k\ge0}$.
An arbitrary strongly $t$-log-convex input $g$ need not admit a
strongly $t$-log-convex backward extension to an $f_0$:
for example, $g_0=1+t$ and $g_k=t^k$ for $k\ge1$ have this property,
but an extension would require
$t f_0(t)-(1+t)^2\succeq_t0$, which is impossible by its
constant coefficient.  In particular, the shifted formulation gives
the type~$A$ preservation theorem directly, without an additional
backward-extension hypothesis.  Any convention at $n=0$ for
$\mathcal Z_n$ is a separate boundary issue.
\end{remark}

\section{Special cases and remarks}\label{sec:special} 

\subsection*{The binomial case \texorpdfstring{\(d=1\)}{d=1}}

For \(d=1\),
\[
\Hog{n}{k}{1}=\binom nk.
\]
Theorem~\ref{thm:main} specializes to the strong
\(t\)-log-convexity-preserving property of the binomial transformation,
one of the examples treated by Zhu and Sun~\cite{ZhuSun2015}.

\subsection*{The Narayana case \texorpdfstring{\(d=2\)}{d=2}}

For \(d=2\),
\[
\Hog{n}{k}{2}
=
\frac{1}{k+1}\binom nk\binom{n+1}{k}
=
N(n+1,k+1).
\]
Thus Theorem~\ref{thm:main} contains the shifted type~$A$ Narayana
transformation as a special case.  Section~\ref{sec:interpolation}
places this kernel in the continuous family $\mathcal N_\theta(n,k)$:
at $\theta=1$ it is the present $d=2$ Hoggatt kernel, whereas at
$\theta=0$ it becomes the type~$B$ Narayana kernel $\binom nk^2$ of
Zhu and Sun~\cite{ZhuSun2015}.

\subsection*{The Baxter case \texorpdfstring{\(d=3\)}{d=3}}

For \(d=3\), the Hoggatt triangle is the Baxter triangle.  Liu and
Mao~\cite{LiuMao2026} proved that the Baxter transformation preserves
ordinary log-convexity.  Theorem~\ref{thm:main} gives the corresponding
strong coefficientwise preservation property and, at the same time,
extends it to all \(d\geq1\).

\subsection*{An application to Bell polynomials}

Let \(B_k(t)=\sum_{j=0}^k S(k,j)t^j\) be the Bell (Touchard)
polynomial, where \(S(k,j)\) is a Stirling number of the second kind.
The sequence \(\{B_k(t)\}_{k\geq0}\) is strongly
\(t\)-log-convex by Chen, Wang and Yang~\cite[Corollary~3.1]{ChenWangYang2011}.
Theorem~\ref{thm:main} therefore gives, for every \(d\geq1\), the
strong \(t\)-log-convexity of the Hoggatt--Bell polynomials
\[
\mathcal B_n^{(d)}(t)
=
\sum_{k=0}^n\Hog{n}{k}{d}B_k(t).
\]

\begin{remark}
The proof of Proposition~\ref{prop:C2} actually establishes a stronger
monotonicity statement.  In the interior range \(l\leq n\), after the
normalization \eqref{eq:Tk}, the quantity \(T_k\) is nonincreasing until
the numerator becomes nonpositive.  This normalized monotonicity may be
useful for other transformations arising from ratios of rising
factorials.
\end{remark}

\section{Concluding remarks and further directions}
\label{sec:conclude}

The continuous Narayana family \(\mathcal N_\theta\) is the common classical specialization.
In addition to the coefficientwise Hankel results, Theorem~\ref{thm:wangyang_intro}
extends the Wang--Yang total-positivity theorems for Narayana triangles
and squares to all $\theta>-1$.
Its row polynomials
carry several simultaneous structures: they are \(\gamma\)-positive,
real-rooted, Stieltjes moment sequences after every specialization
\(t\geq0\), and coefficientwise Hankel-totally positive before
specialization.  The last property is an all-orders strengthening of
strong \(t\)-log-convexity.  The transformation theorem then shows that
this same continuous family preserves strong \(t\)-log-convexity.
At the endpoint \(\theta=1\), the kernel is the \(d=2\) Hoggatt triangle,
and Theorem~\ref{thm:main} extends the preservation phenomenon in the
transverse direction to all integers \(d\geq1\).

The Hoggatt proof separates the preservation problem into two independent
pieces.  The global positivity of the Hoggatt row-generating polynomials
is supplied by the Schur-positive argument of Chern and Shi, while the
mixed one-sign-change property ultimately rests on the hypergeometric
hitting inequality of Lemma~\ref{lem:hypergeom_hit}.

The classical Gegenbauer approach also gives constructive information
about minors of arbitrary order. Their extreme coefficients are determined
by Theorem~\ref{ar:thm:extreme}, while the normalized minors need not be
$\gamma$-positive (Remark~\ref{ar:rem:minors-not-gamma}).
A natural open problem is an explicit coefficientwise-positive formula
for their interior coefficients. A planar-network interpretation of the
hypergeometric hitting inequality could yield broader
positivity-preserving transformations. A separate manuscript
\cite{ZengQUltraspherical} investigates two $q$-analogues of the
present interpolation; no $q$-deformation is used in the proofs here.

\subsection*{A Gamma-shift extension}

A natural direction is to replace the consecutive shifts
\(1,2,\ldots,d\) by other shifts.  For
\(\boldsymbol{\alpha}=(\alpha_1,\ldots,\alpha_s)\in[1,\infty)^s\), define
\[
A_{\boldsymbol{\alpha}}(n,k)
=
\prod_{j=1}^s
\frac{\Gamma(n+\alpha_j)\Gamma(\alpha_j)}
     {\Gamma(k+\alpha_j)\Gamma(n-k+\alpha_j)},
\qquad 0\leq k\leq n.
\]
The Hoggatt triangle is recovered from
\(\boldsymbol{\alpha}=(1,2,\ldots,d)\).  The interpolation of
Section~\ref{sec:interpolation} is exactly the one-dimensional ray
\[
\boldsymbol{\alpha}=(1,1+\theta),
\qquad \theta\ge0,
\]
so Theorem~\ref{thm:interpolation} completely settles this family.

It remains natural to ask for which general parameter vectors
\(\boldsymbol{\alpha}\) the row-generating polynomials
\[
\sum_{k=0}^n A_{\boldsymbol{\alpha}}(n,k)t^k
\]
are strongly \(t\)-log-convex and for which parameters the associated
linear transformation preserves strong \(t\)-log-convexity.  Our proof
suggests that the decisive issue is an analogue of
Lemma~\ref{lem:hypergeom_hit} for
\[
R_{\boldsymbol{\alpha}}(x)
=
\prod_{j=1}^s (x+\alpha_j-1).
\]
The family $\boldsymbol{\alpha}=(1,1+\theta)$ shows that nonintegral
shifts can satisfy the required inequalities.  We leave the
characterization of all admissible shift vectors as an open problem.

\end{document}